\documentclass{amsart}
\usepackage[utf8]{inputenc}
\usepackage{hyperref}
\usepackage{amsfonts}
\usepackage{amssymb}
\usepackage{amsthm}
\usepackage{bbold}
\usepackage{tikz}
\usetikzlibrary{matrix,arrows,decorations.pathmorphing, positioning, automata}

\DeclareMathOperator{\T}{\mathsf T}

\DeclareMathOperator{\F}{\mathsf F}

\DeclareMathOperator{\e}{\varepsilon}
\DeclareMathOperator{\Db}{\mathsf D^{\mathsf b}}
\DeclareMathOperator{\D1}{\mathsf D_\text{$1$}}
\DeclareMathOperator{\Dn}{\mathsf D_\text{$n$}}

\DeclareMathOperator{\K1}{\mathsf K_\text{$1$}}

\DeclareMathOperator{\Kn}{\mathsf K_\text{$n$}}
\DeclareMathOperator{\Cb}{\mathsf C^{\mathsf b}}
\DeclareMathOperator{\C1}{\mathsf C_\text{$1$}}
\DeclareMathOperator{\Cn}{\mathsf C_\text{$n$}}

\DeclareMathOperator{\Hom}{\mathsf{Hom}}

\DeclareMathOperator{\End}{\mathsf{End}}

\DeclareMathOperator{\modules}{\mathsf{mod}}

\DeclareMathOperator{\proj}{\mathsf{proj}}

\DeclareMathOperator{\rad}{\mathsf{rad}}

\makeatletter
\def\amsbb{\use@mathgroup \M@U \symAMSb}
\makeatother

\newtheorem{theorem}[equation]{Theorem}
\newtheorem{lemma}[equation]{Lemma}

\newtheorem{proposition}[equation]{Proposition}

\newtheorem{criterion}[equation]{Criterion}
\newtheorem*{theorem*}{Theorem}
\newtheorem*{lemma*}{Lemma}
\newtheorem*{corollary*}{Corollary}
\newtheorem*{proposition*}{Proposition}
\newtheorem*{observation*}{Observation}
\newtheoremstyle{named}{}{}{\itshape}{}{\bf}{.}{.5em}{\thmnote{#3's }#1}
\theoremstyle{named}

\theoremstyle{remark}
\newtheorem*{remark}{Remark}
\theoremstyle{definition}
\newtheorem*{example}{Example}

\title[Quadratic monomial algebras]{Differential modules over \\ quadratic monomial algebras}
\author{Torkil Stai}
\address{Institutt for matematiske fag, NTNU, 7491 Trondheim, Norway}
\email{torkil.stai@math.ntnu.no}

\begin{document}

\begin{abstract}
   We compare the so-called clock condition to the gradability of certain differential modules over quadratic monomial algebras. For a stably hereditary algebra or a gentle one-cycle algebra, these considerations show that the orbit category of its bounded derived category with respect to a positive power of the shift functor is triangulated if and only if the algebra is piecewise hereditary.
\end{abstract}

\maketitle

\section{Introduction} \label{section:intro}
Let $\T$ be a category with an automorphism $\F$. The \textit{orbit category} $\T\!/\F$ has the objects of $\T$ and morphism spaces given by
\[
\T\!/\F(X,Y)= \bigoplus_{i \in \amsbb Z} \T(X, \F^i Y)
\]
with the natural composition. Suppose now that $\T$ is a triangulated category and that $\F$ is exact. Does $\T\!/\F$ inherit a triangulated structure so that the canonical projection $\T \to \T\!/\F$ becomes exact?

In such vast generality it is not clear how to even look for an answer. As a partial remedy, Keller showed in the seminal \cite{MR2184464} that certain orbit categories of derived categories of algebras admit an embedding into a triangulated hull with a universal property. This allows us to rephrase the above question in these cases as `does the orbit category coincide with its triangulated hull?' and Keller moreover proved that whenever the algebra is piecewise hereditary, the answer is affirmative. Intriguingly, there are no known counter-examples to the converse of the last result, and in \cite{MR3024264} it was conjectured that (a $\tau_2$-finite algebra) $\Lambda$ must be piecewise hereditary in order for the category $\Db(\modules \Lambda)/\amsbb S \circ \Sigma^{-2}$ to be triangulated. This problem remains open, indicating that the business of triangulated hulls is a delicate one.

Our humble strategy is to contribute by attacking a baby case, in the following sense. Powers of the shift functor itself are certainly comprehensible automorphisms of $\Db(\modules \Lambda)$, and we naively hope that this makes it feasible to understand when $\Db(\modules \Lambda)/\Sigma^n$ is triangulated for a positive integer $n$. Our first result shows that this problem is invariant under $n$.

\begin{theorem*}[See Theorem \ref{thm:homogeneity}]
   If the orbit category $\Db(\modules \Lambda)/\Sigma^n$ is triangulated for one choice of $n$, then it is triangulated for each $n$.
\end{theorem*}

In \cite{P1} the embedding of this orbit category into its triangulated hull was made explicit in the terms of $n$-periodic complexes. Moreover, evidence was collected to support the conjecture that also this orbit category is triangulated if and only if $\Lambda$ is piecewise hereditary. We will turn a fraction of this guesswork into the below theorem. Recall that a gentle one-cycle algebra is simply a gentle algebra whose ordinary quiver contains precisely one cycle. On the other hand, the class of stably hereditary algebras contains those that are stably equivalent to a hereditary one, and in particular the radical square zero algebras by \cite[X.2]{MR1476671}. Hence, although a modest contribution, our result might be seen as a promising first step towards a full resolution. Indeed, recall that in \cite{MR2413349} the piecewise hereditary algebras were described as those algebras whose strong global dimension is finite. This beautiful characterization, however, was first achieved for radical square zero algebras in \cite{MR2041672}.

\begin{theorem*}[See Theorem \ref{thm:stablyhereditary} and Theorem \ref{thm:gentleonecycle}]
   Let $\Lambda$ be a stably hereditary or a gentle one-cycle algebra and let $n$ be a positive integer. The orbit category $\Db(\modules \Lambda)/\Sigma^n$ is triangulated if and only if $\Lambda$ is piecewise hereditary.
\end{theorem*}

We prove the missing implication by consulting the hands-on description of $\Db(\modules \Lambda)/\Sigma$ and its triangulated hull from \cite{P1}. This allows us to construct, over an arbitrary quadratic monomial algebra, an object in the triangulated hull which cannot belong to the orbit category, under the assumption that $\Lambda$ contravenes the so-called clock condition. The argument is then completed by results of \cite{MR892057,MR2387592} which reveal that if the $\Lambda$ in our theorem is not piecewise hereditary, then it does violate said condition.

\subsection{Overview, conventions and acknowledgements} This short note starts with a reminder on periodic complexes and the triangulated hull, which is used to prove the first above cited theorem. The clock condition is then discussed, and in particular reformulated for radical square zero algebras (Lemma \ref{lem:radicalsquarezero}), before it is employed in the construction of the pivotal object of the triangulated hull (Proposition \ref{prop:clock}). In the last section the stably hereditary and gentle (one-cycle) algebras are introduced, and the second above cited theorem is proved.

We denote by $\Lambda$ a basic finite-dimensional algebra of finite global dimension over an algebraically closed field $\mathbb k$. All $\Lambda$-modules are finitely generated right modules, and these form the category $\modules \Lambda$ with the subcategory $\proj \Lambda$ of projectives.

The author would like to thank Ragnar-Olaf Buchweitz for initially inspiring him to pursue the conjecture in question for radical square zero algebras. He is also most grateful to Steffen Oppermann for numerous helpful discussions.

\section{The triangulated hull} \label{section:thehull}
Let $n$ be a positive integer. The automorphism $\Sigma^n \cong - \otimes_{\Lambda} \Sigma^n \Lambda$ of $\Db(\modules \Lambda)$ satisfies the hypothesis in \cite{MR2184464} that ensure the existence of a triangulated hull of the orbit category $\Db(\modules \Lambda)/\Sigma^n$. By the above discussion, declaring that the latter is triangulated is the same as saying that it coincides with its triangulated hull.

\subsection{Periodic complexes} \label{subsec:periodiccomplexes} Denote by $\Cn(\modules \Lambda)$ the category of $n$-periodic complexes and chain maps over $\Lambda$. To this one can associate the $n$-periodic  homotopy (or stable) category $\Kn(\modules \Lambda)$ and derived category $\Dn(\modules \Lambda)$, obtained by formally inverting the class of quasi-isomorphisms. These are both triangulated, and projective resolutions provide a convenient triangle equivalence $\Dn(\modules \Lambda) \cong \Kn(\proj \Lambda)$.

There is a forgetful functor $\Cb(\modules \Lambda) \to \Cn(\modules \Lambda)$, and objects in its essential image are referred to as \textit{gradable}. Explicitly, forgetting means assigning to
\[
   0 \to X^0 \to X^1 \to \cdots \to X^{l-1} \to X^l \to 0
\]
the $n$-periodic
\[
   \cdots \to \mathop{\oplus}_{i \equiv 1} X^i  \to \mathop{\oplus}_{i \equiv 2}X^i \to \cdots \to \mathop{\oplus}_{i \equiv 0}X^i \to \mathop{\oplus}_{i \equiv 1}X^i \to \cdots
\]
where the differentials are obvious, and each congruence is taken modulo $n$. This functor is exact, and so the same procedure automatically gives a functor
\[
   \Delta_n \colon \Db(\modules \Lambda) \to \Dn(\modules \Lambda).
\]
Moreover, the latter induces a fully faithful functor $\Db(\modules \Lambda)/\Sigma^n \to \Dn(\modules \Lambda)$, and in \cite{P1} it is shown that this realizes the embedding of the orbit category into its triangulated hull.

\subsection{Homogeneity} \label{subsec:homogeneity} Morally, whether $\Db(\modules \Lambda)/\Sigma^n$ is triangulated or not, should be independent of $n$. Let us put this intuition on formal footing.

\begin{theorem} \label{thm:homogeneity}
   If the orbit category $\Db(\modules \Lambda)/\Sigma^n$ is triangulated for one choice of $n$, then it is triangulated for each $n$.
\end{theorem}

\begin{proof}
   It suffices to show that $\Delta_1 \colon \Db(\modules \Lambda) \to \D1(\modules \Lambda)$ is dense if and only if $\Delta_n \colon \Db(\modules \Lambda) \to \Dn(\modules \Lambda)$ is dense. Given an $n$-periodic complex we construct a $1$-periodic one by taking the coproduct of $\Lambda$-modules over one period and equipping this with the obvious induced differential. There is an evident action on morphisms yielding a functor $\Delta_{n,1} \colon \Dn(\modules \Lambda) \to \D1(\modules \Lambda)$ which moreover satisfies $\Delta_{n,1} \Delta_n \cong \Delta_1$. On the other hand, viewing $1$-periodic complexes as $n$-periodic gives a functor $\iota \colon \D1(\modules \Lambda) \to \Dn(\modules \Lambda)$.
   \begin{center}
      \begin{tikzpicture}\matrix(a)[matrix of math nodes,
      row sep=3em, column sep=1em,
      text height=1.5ex, text depth=0.25ex]
      { & \Db(\modules \Lambda) &  \\
      \Dn(\modules \Lambda) && \D1(\modules \Lambda) \\ };
         \path[->,font=\scriptsize] (a-1-2) edge node[above left]{$\Delta_n$} (a-2-1)
         (a-1-2) edge node[above right]{$\Delta_1$} (a-2-3);
         \path[transform canvas={yshift=.3ex},->,font=\scriptsize] (a-2-1) edge node[above]{$\Delta_{n,1}$} (a-2-3);
         \path[transform canvas={yshift=-.3ex},->,font=\scriptsize](a-2-3) edge node[below]{$\iota$} (a-2-1);
      \end{tikzpicture}
   \end{center}
   Assume that $\Delta_n$ is dense, and let $Y \in \D1(\modules \Lambda)$ be indecomposable. Then there is some $X=X_1 \oplus \cdots \oplus X_m \in \Db(\modules \Lambda)$ with each $X_i$ indecomposable such that $\iota(Y) \cong \Delta_n(X)$, to which applying $\Delta_{n,1}$ gives
   \[
      \Delta_{n,1} \iota (Y) \cong \Delta_1(X) \cong \bigoplus_{i=1}^m \Delta_1(X_i).
   \]
   $Y$ is clearly a summand of the left hand side, and each $\Delta_1(X_i)$ is indecomposable by \cite[Corollary 1.3]{Farnsteiner}. As a decomposition into indecomposables is essentially unique, it follows that $Y$ belongs to the essential image of $\Delta_1$. For the converse, a similar argument will do. Namely, assume that $\Delta_1$ is dense, and pick $Y \in \Dn(\modules \Lambda)$. Then we find $X \in \Db(\modules \Lambda)$ such that $\Delta_{n,1}(Y) \cong \Delta_1(X)$, and hence $\iota \Delta_{n,1}(Y) \cong \iota \Delta_1(X)$ in $\Dn(\modules \Lambda)$. It is clear that $Y$ is a summand of $\iota \Delta_{n,1}(Y)$, so it suffices to show that the essential image of $\Delta_n$ contains $\iota \Delta_1(X)$. But the latter is the $n$-periodic
   \[
      \cdots \to \mathop{\oplus}_{i \in \amsbb Z} X^i \xrightarrow{\e} \mathop{\oplus}_{i \in \amsbb Z} X^i \xrightarrow{\e} \mathop{\oplus}_{i \in \amsbb Z} X^i \to \cdots
   \]
   where $\e$ can be written as a matrix whose only non-zero entries lie on the first sub-diagonal. This feature of $\e$ allows $\iota \Delta_1(X)$ to decompose in $\Dn(\modules \Lambda)$ as

   \begin{align*}
      &  \cdots \to \mathop{\oplus}_{i \equiv 0}X^i \xrightarrow{\bar{\e}} \mathop{\oplus}_{i \equiv 1}X^i \to \cdots \to \mathop{\oplus}_{i \equiv -1 }X^i \xrightarrow{\bar{\e}} \mathop{\oplus}_{i \equiv 0}X^i \to \cdots \\
      & \makebox[0pt][l]{
    \vphantom{$\bigoplus$}
    \ooalign{\phantom{$\cdots\cdots\cdots\cdots\cdots\cdots\cdots\cdots\cdots\cdots\cdots\cdots\cdots\cdots\cdots\cdots\cdots\cdots $}\cr\hss$\bigoplus$\hss}}     \\
    &  \cdots \to \mathop{\oplus}_{i \equiv 1 }X^i \xrightarrow{\bar{\e}} \mathop{\oplus}_{i \equiv 2 }X^i \to \cdots \to \mathop{\oplus}_{i \equiv 0 }X^i \xrightarrow{\bar{\e}} \mathop{\oplus}_{i \equiv 1}X^i \to \cdots \\
    & \makebox[0pt][l]{
    \vphantom{$\oplus$}
    \ooalign{\phantom{$\cdots\cdots\cdots\cdots\cdots\cdots\cdots\cdots\cdots\cdots\cdots\cdots\cdots\cdots\cdots\cdots\cdots\cdots$}\cr\hss$\bigoplus$\hss}}    \\
    & \makebox[0pt][l]{
    \vphantom{$\vdots$}
    \ooalign{\phantom{$\cdots\cdots\cdots\cdots\cdots\cdots\cdots\cdots\cdots\cdots\cdots\cdots\cdots\cdots\cdots\cdots\cdots\cdots$}\cr\hss$\vdots$\hss}}    \\
    & \makebox[0pt][l]{
    \vphantom{$\oplus$}
    \ooalign{\phantom{$\cdots\cdots\cdots\cdots\cdots\cdots\cdots\cdots\cdots\cdots\cdots\cdots\cdots\cdots\cdots\cdots\cdots\cdots$}\cr\hss$\bigoplus$\hss}}    \\
     &  \cdots \to \mathop{\oplus}_{i \equiv -1 }X^i \xrightarrow{\bar{\e}} \mathop{\oplus}_{i \equiv 0 }X^i \to \cdots \to \mathop{\oplus}_{i \equiv -2 }X^i \xrightarrow{\bar{\e}} \mathop{\oplus}_{i \equiv -1 }X^i \to \cdots
   \end{align*}
   where each $\bar{\e}$ denotes the obvious restriction of $\e$. These summands clearly all belong to the essential image of $\Delta_n$.
\end{proof}

\subsection{Differential modules} \label{subsec:differentialmodules}
Following \cite{MR2308849}, by a \textit{differential module} over $\Lambda$ we mean a module over the algebra $\Lambda[\e]= \Lambda[t]/(t^2)$ of dual numbers, that is a pair $(M, \e_M)$ with \textit{underlying module} $M \in \modules \Lambda$ and \textit{differential}  $\e_M \in \End_{\Lambda}(M)$ squaring to zero. If the underlying module belongs to $\proj \Lambda$ then such an object is called \textit{relatively projective}. Observe that a differential module is nothing but a $1$-periodic complex, and by Theorem \ref{thm:homogeneity} we may restrict our attention to the gradability of such objects. Indeed, if we want to show that the orbit category $\Db(\modules \Lambda)/\Sigma^n$ is strictly smaller than its triangulated hull then we need only point to an object in $\K1(\proj \Lambda)$ that is not gradable. What is more, it suffices to demonstrate this property in $\C1(\proj \Lambda)$, since passing to the homotopy category preserves non-gradability. To see why this is the case, take any $X \in \C1(\modules \Lambda)$ and assume it is gradable in $\K1(\modules \Lambda)$. Then there is some null-homotopic $Y$ such that $X \oplus Y$ lies in the essential image of the forgetful functor. It now follows that $X$ is gradable already in $\C1(\modules \Lambda)$, since said image is closed under direct summands (see for instance \cite[Proposition 1.4]{Farnsteiner}). Let us gather up the essence of the current section.

\begin{criterion} \label{crit:nongradable}
   If $\Lambda$ admits a relatively projective differential module which is non-gradable, then $\Db(\modules \Lambda)/\Sigma^n$ is not triangulated for any positive integer $n$.
\end{criterion}

\section{The clock condition} \label{section:clock}

Our algebra $\Lambda$ is of the form $\mathbb k Q/I$ for some quiver $Q$ and admissible ideal $I$. If $\Lambda$ is \textit{monomial}, meaning that $I$ is generated by paths, then it is said to satisfy the \textit{clock condition} if in each cycle of $Q$, the number of clockwise oriented relations equals the number of counter-clockwise oriented ones.

\subsection{Radical square zero algebras}
Suppose that $(\rad\Lambda)^2$ vanishes. In this case the clock condition says precisely that in each cycle of $Q$, the number of clockwise oriented arrows equals the number of counter-clockwise oriented ones. This notion has been used in both \cite{MR2497950} (under the alias `walk condition') and in \cite{BautistaLiu} (here referred to as $Q$ being `gradable') to describe the derived categories of radical square zero algebras. Before putting it to use in the context of differential modules, we offer the following interpretation.

As always, $\Lambda$ is trivially a graded algebra $\Lambda^0$ concentrated in degree zero. We can also view it as the graded algebra $\Lambda^\Pi$ by letting the paths be homogeneous elements of degree equal to their lengths. Two graded algebras are \textit{graded equivalent} if there is an equivalence between the respective categories of graded modules.

\begin{observation*}
   A radical square zero algebra $\Lambda$ satisfies the clock condition if and only if the graded algebras $\Lambda^\Pi$ and $\Lambda^0$ are graded equivalent.
\end{observation*}

\begin{proof}
   Let $P_1,\dots ,P_n$ be the indecomposable projective $\Lambda$-modules. For a graded module $M$ and an integer $s$, let $M(s)$ denote the $s$'th graded shift of $M$. From `graded Morita theory' (see e.g.\ \cite[Theorem 5.4]{MR659212}) we infer that $\Lambda^\Pi$ and $\Lambda^0$ are graded equivalent if and only if there are integers $s_i$ such that
   \[
   \Hom_{\Lambda^\Pi}(P_i(s_i), P_j(s_j)) \cong \Hom_{\Lambda}(P_i, P_j)
   \]
   for all $1 \leq i,j \leq n$. As a morphism of graded modules is required to be homogeneous of degree zero, it is clear that such $s_i$ are available if and only if each cycle of $Q$ has the same number of clockwise as counter-clockwise oriented arrows.
\end{proof}

In the sequel we will say that an element of $\Lambda$ of the form
\[
\e = \sum_{\alpha \in Q_1} k_{\alpha} \alpha
\]
is \textit{generic} if each $k_{\alpha} \in \mathbb k$ is non-zero. Since $\e \in \End_{\Lambda}(\Lambda)$ squares to zero we may consider the relatively projective differential $\Lambda$-module $(\Lambda, \e)$, which leads to yet another way of expressing the clock condition.

\begin{lemma} \label{lem:radicalsquarezero}
   Let $\Lambda$ be a radical square zero algebra and pick a generic element $\e \in \Lambda$. Then $(\Lambda, \e)$ is gradable if and only if $\Lambda$ satisfies the clock condition.
\end{lemma}

\begin{proof}
   Assume that $(\Lambda, \e)$ is gradable. This means there is some decomposition
   \[
   \Lambda = \bigoplus_{i=1}^n X_i
   \]
   in $\modules \Lambda$ such that the component $X_i \xrightarrow{\iota} \Lambda \xrightarrow{\e} \Lambda \xrightarrow{\pi} X_j$ vanishes unless $j=i+1$. Denote, for each vertex $a$ of $Q$, by $P_a$ the corresponding indecomposable projective. Then for each arrow $a \to b$ in $Q$ it is evident that $P_a \mid X_i$ implies $P_b \mid X_{i+1}$, and moreover that this forces each cycle of $Q$ to have the same number of clockwise as counter-clockwise oriented arrows. Conversely, when $\Lambda$ satisfies the clock condition there is an obvious way of equipping $(\Lambda, \e)$ with a grading (which is moreover unique up to shifting each summand by the same integer when $\Lambda$ is connected).
\end{proof}

\subsection{Quadratic monomial algebras}

A monomial algebra $\Lambda$ is said to be \textit{quadratic} if the ideal $I$ is generated by paths of length $2$. In particular, the radical square zero algebras are of this form, and the idea of Lemma \ref{lem:radicalsquarezero} extends as follows.

\begin{proposition} \label{prop:clock}
   Let $\Lambda$ be a quadratic monomial algebra. If $\Lambda$ violates the clock condition, then it admits a non-gradable and relatively projective differential module.
\end{proposition}

\begin{proof}
   Let $C$ be a cycle of $Q$ in which the number of clockwise oriented relations differs from the number of counter-clockwise ones. Clearly, $C$ is determined by its `chords', i.e.\ intervals of maximal length in which the arrows are all either clockwise or counter-clockwise oriented. Let
   \[
      G = (i \xrightarrow{\alpha_i} i+1 \xrightarrow{\alpha_{i+1}} i+2 \to \cdots \to j-1 \xrightarrow{\alpha_{j-1}} j)
   \]
   be such a chord in $C$, and construct another linearly ordered interval $G'$ as follows. If $\alpha_{l+1}\alpha_l \in I$ for each $i \leq l \leq j-2$, then $G'=G$. If not, there is a minimal such $l$ with the property that $\alpha_{l+1}\alpha_l \notin I$, and we replace the sequence $l \xrightarrow{\alpha_l} l+1 \xrightarrow{\alpha_{l+1}} l+2$ in $G$ by an arrow $l \xrightarrow{\alpha_{l+1}\alpha_l} l+2$. After finitely many iterations, we have on our hands a linearly ordered $G'$ in which no arrow belongs to $I$, but where the composition of any pair of neighboring arrows does lie in $I$. The collection of these $G'$, one for each chord $G$ in $C$, form in an obvious way a new cycle $C'$.

   The evident merit of this construction is that the algebra $\Lambda'$ with ordinary quiver $C'$ and subject to the relations on the original cycle $C$, becomes a radical square zero algebra. Observe next that a chord $G$ in $C$ involving $n \geq 0$ relations gives rise to a chord $G'$ having $n+1$ arrows, and hence the number of clockwise arrows in $C'$ differs from the number of counter-clockwise ones. Lemma \ref{lem:radicalsquarezero} thus shows that $(\Lambda', \e')$ is a non-gradable differential $\Lambda'$-module for each generic $\e' \in \Lambda'$. It is straightforward to translate this data to a relatively projective differential $\Lambda$-module which moreover cannot be gradable.
\end{proof}

\begin{example}
   Let $\Lambda$ be the quadratic monomial algebra given by
   
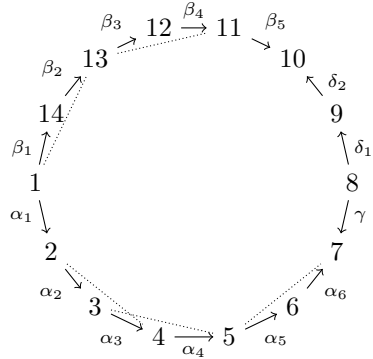
\begin{center}
      \begin{tikzpicture}
         \foreach \a in {1,2,...,14}{
         \draw (\a*360/14+154.3: 6em) node(\a){\a};
         }
         \path[->, font=\scriptsize] (1) edge node[left]{$\alpha_1$}(2)
         (2) edge node[below left]{$\alpha_2$} (3)
         (3) edge node[below, shift={(-.2,-.05)}]{$\alpha_3$}(4)
         (4) edge node[below]{$\alpha_4$}(5)
         (5) edge node[below, shift={(.2,-.05)}]{$\alpha_5$}(6)
         (6) edge node[below right]{$\alpha_6$}(7)

         (8) edge node[right]{$\delta_1$}(9)
         (9) edge node[right, shift={(.05,.05)}]{$\delta_2$}(10)

         (8) edge node[right]{$\gamma$}(7)

         (1) edge node[left]{$\beta_1$}(14)
         (14) edge node[above left]{$\beta_2$}(13)
         (13) edge node[above, shift={(-.2,.05)}]{$\beta_3$}(12)
         (12) edge node[above]{$\beta_4$}(11)
         (11) edge node[above, shift={(.2,.05)}]{$\beta_5$}(10);
         \path[densely dotted] (2) edge (4)
         (3) edge (5)
         (5) edge (7)
         (1) edge (13)
         (13) edge (11);
         \end{tikzpicture}
   \end{center}
   with the indicated relations. Then the quiver
   
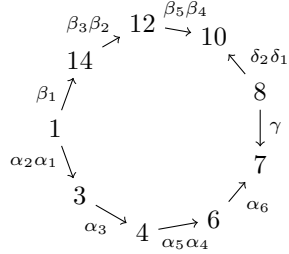
\begin{center}
      \begin{tikzpicture}
         \draw (180: 4em) node(1){1}
         (220: 4em) node(3){3}
         (260: 4em) node(4){4}
         (300: 4em) node(6){6}
         (340: 4em) node(7){7}
         (20: 4em) node(8){8}
         (60: 4em) node(10){10}
         (100: 4em) node(12){12}
         (140: 4em) node(14){14};
         \path[->, font=\scriptsize] (1) edge node[left]{$\alpha_2\alpha_1$}(3)
         (3) edge node[below, shift={(-.2,-.0)}]{$\alpha_3$}(4)
         (4) edge node[below, shift={(.1,-.05)}]{$\alpha_5\alpha_4$}(6)
         (6) edge node[below right]{$\alpha_6$}(7)

         (8) edge node[right]{$\gamma$}(7)

         (8) edge node[right, shift={(.05,.1)}]{$\delta_2\delta_1$}(10)

         (1) edge node[left]{$\beta_1$}(14)
         (14) edge node[above, shift={(-.3,.0)}]{$\beta_3\beta_2$}(12)
         (12) edge node[above, shift={(.1,.05)}]{$\beta_5\beta_4$}(10);
         \end{tikzpicture}
   \end{center}
   determines a radical square zero algebra $\Lambda'$, and $(\Lambda', \e')$ is non-gradable for each generic $\e' \in \Lambda'$. In terms of our original algebra, this means in particular that
   \[
   P=P_1 \oplus P_3 \oplus P_4 \oplus P_6 \oplus P_7 \oplus P_8 \oplus P_{10} \oplus P_{12} \oplus P_{14} \in \proj \Lambda
   \]
   together with
   \[
   \e_P = \alpha_2\alpha_1 + \alpha_3 + \alpha_5\alpha_4 + \alpha_6 + \beta_1 + \beta_3\beta_2 + \beta_5\beta_4 + \gamma + \delta_2\delta_1 \in \End_{\Lambda}(P)
   \]
   defines a relatively projective and non-gradable differential $\Lambda$-module $(P, \e_P)$.
\end{example}

\begin{remark}
   It seems plausible that a similar approach can be used to devise non-gradable differential modules over arbitrary monomial algebras violating the clock condition. For instance, to an algebra whose ordinary quiver contains the cycle
   
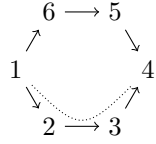
\begin{center}
      \begin{tikzpicture}
         \foreach \a in {1,2,...,6}{
         \draw (\a*360/6+120: 2.5em) node(\a){\a};
         }
         \path[->] (1) edge (2)
         (2) edge (3)
         (3) edge (4)
         (1) edge (6)
         (6) edge (5)
         (5) edge (4);
         \path[in=225,out=315, looseness=1.7, densely dotted] (1) edge (4);
      \end{tikzpicture}
   \end{center}
   one can associate the radical square zero algebra whose ordinary quiver is
   \begin{center}
   \begin{tikzpicture}
      \matrix(a)[matrix of math nodes,
      row sep=2.5em, column sep=2em,
      text height=1.5ex, text depth=0.25ex]
      { 1 & 3 & 4 \\};
      \path[->](a-1-1) edge (a-1-2)
      (a-1-2) edge (a-1-3);
      \path[->, bend left](a-1-1) edge (a-1-3);
   \end{tikzpicture}
   \end{center}
   and refer again to Lemma \ref{lem:radicalsquarezero}. For the sake of brevity, and since the applications we have in mind are all quadratic, we will not dwell further on such a strategy.
\end{remark}

\section{Applications} \label{section:applications}

\subsection{Stably hereditary algebras} \label{subsec:stablyhereditary}

Recall that an algebra is \textit{stably hereditary} if indecomposable submodules of indecomposable projectives are projective or simple, and also indecomposable factor modules of indecomposable injectives are injective or simple (projective-injectives are only required to satisfy one of these conditions, however). The class of such algebras contains those that are stably equivalent to a hereditary one, but these are not all. One example of a stably hereditary algebra which is not stably equivalent to a hereditary one, originating from \cite{MR0404243}, is $\End_{\mathbb k \amsbb A_3}(\mathbb k \amsbb A_3 \oplus S)$ where $\amsbb A_3$ is linearly ordered and $S$ is the simple module which is neither projective nor injective. More generally, \cite[Proposition 3.7]{Xi2002193} provides a recipe for generating new stably hereditary algebras from previous ones.

\begin{theorem}\label{thm:stablyhereditary}
   Let $\Lambda$ be a stably hereditary algebra and let $n$ be a positive integer. The orbit category $\Db(\modules \Lambda)/\Sigma^n$ is triangulated if and only if $\Lambda$ is piecewise hereditary.
\end{theorem}
\begin{proof}
   By \cite[Theorem 1]{MR2184464}, the orbit category is triangulated when $\Lambda$ is piecewise hereditary. For the converse, note first that $\Lambda$ is quadratic monomial by \cite{MR532393}. If it is not piecewise hereditary, and hence not iterated tilted, then it violates the clock condition by \cite[Theorem 3.2.5]{MR2387592} and hence admits a non-gradable and relatively projective differential $\Lambda$-module (Proposition \ref{prop:clock}), which suffices (Criterion \ref{crit:nongradable}).
\end{proof}

\subsection{Gentle (one-cycle) algebras} \label{subsec:gentleonecycle}

We say that a quadratic monomial algebra $\Lambda$ is \textit{gentle} if the pair $(Q,I)$ adheres to the following conditions.
\begin{enumerate}
   \item At any vertex, there are at most two incoming and at most two outgoing arrows.
   \item For each arrow $\beta$, there is at most one arrow $\alpha$ such that $0 \neq \alpha \beta \in I$ and at most one arrow $\gamma$ such that $0 \neq \beta \gamma \in I$.
   \item For each arrow $\beta$, there is at most one arrow $\alpha$ such that $\alpha \beta \notin I$ and at most one arrow $\gamma$ such that $\beta \gamma \notin I$.
\end{enumerate}
The gentle algebras have received a great deal of interest, and are consequently fairly well understood. For instance, they are more conseptually characterized as those algebras whose repetitive algebra is special biserial (see \cite{MR892057,MR1102821}). If $(Q,I)$ satisfies the above $(1)$ through $(3)$ and, in addition, $Q$ contains a unique cycle, then we imaginatively say that $\Lambda$ is a \textit{gentle one-cycle algebra}.

\begin{theorem} \label{thm:gentleonecycle}
   Let $\Lambda$ be a gentle one-cycle algebra and let $n$ be a positive integer. The orbit category $\Db(\modules \Lambda)/\Sigma^n$ is triangulated if and only if $\Lambda$ is piecewise hereditary.
\end{theorem}
\begin{proof}
   \cite[Theorem 1]{MR2184464} settles one implication. For what remains, assume that $\Lambda$ is not piecewise hereditary. In particular it is not iterated tilted, and thus violates the clock condition by \cite[Theorem A]{MR892057}. Hence $\Lambda$ admits a non-gradable and relatively projective differential module (Proposition \ref{prop:clock}), which suffices (Criterion \ref{crit:nongradable}).
\end{proof}

\bibliographystyle{amsplain}
\bibliography{main}
\end{document}